\documentclass[reqno]{amsart}

\usepackage{amsbsy}

\usepackage{times}
\usepackage[T1]{fontenc}
\usepackage{mathrsfs}
\usepackage{latexsym}
\usepackage[dvips]{graphics}
\usepackage{epsfig}
\usepackage{flafter,epsf}
\usepackage{amsmath,amsfonts,amsthm,amssymb,amscd}
\usepackage{color}
\usepackage{fix-cm}
\usepackage{hyperref}
\hypersetup{colorlinks=true,linkcolor=blue}
\usepackage{color}
\usepackage{url}
\newcommand{\bburl}[1]{\textcolor{blue}{\url{#1}}}

\usepackage{comment}
\usepackage{multirow}

\newcommand\be{\begin{equation}}
\newcommand\ee{\end{equation}}
\newcommand\bea{\begin{eqnarray}}
\newcommand\eea{\end{eqnarray}}
\newcommand\bi{\begin{itemize}}
\newcommand\ei{\end{itemize}}
\newcommand\ben{\begin{enumerate}}
\newcommand\een{\end{enumerate}}

\newtheorem{thm}{Theorem}[section]

\newtheorem{lem}[thm]{Lemma}

\theoremstyle{definition}

\theoremstyle{definition}

\theoremstyle{definition}

\newcommand{\twocase}[5]{#1 \begin{cases} #2 & \text{#3}\\ #4
&\text{#5} \end{cases}   }

\newcommand\cycle[2][\,]{%
  \readlist\thecycle{#2}%
  (\foreachitem\i\in\thecycle{\ifnum\icnt=1\else#1\fi\i})%
}

\newcommand{\abs}[1]{\left|#1\right|}

\numberwithin{equation}{section}

\makeatletter
\def\@tocline#1#2#3#4#5#6#7{\relax
  \ifnum #1>\c@tocdepth 
  \else
    \par \addpenalty\@secpenalty\addvspace{#2}%
    \begingroup \hyphenpenalty\@M
    \@ifempty{#4}{%
      \@tempdima\csname r@tocindent\number#1\endcsname\relax
    }{%
      \@tempdima#4\relax
    }%
    \parindent\z@ \leftskip#3\relax \advance\leftskip\@tempdima\relax
    \rightskip\@pnumwidth plus4em \parfillskip-\@pnumwidth
    #5\leavevmode\hskip-\@tempdima
      \ifcase #1
       \or\or \hskip 1em \or \hskip 2em \else \hskip 3em \fi%
      #6\nobreak\relax
    \hfill\hbox to\@pnumwidth{\@tocpagenum{#7}}\par
    \nobreak
    \endgroup
  \fi}
\makeatother

\begin{document}

\title{The Real Schwarz Lemma: The Sequel}

\author[Baily]{Benjamin Baily}
\email{bmb2@williams.edu}
\address{Department of Mathematics, Williams College, MA 01267}

\author[Geller]{Jonathan Geller}
\email{jmg8@williams.edu}
\address{Department of Mathematics, Williams College, MA 01267}

\author[Miller]{Steven J. Miller}
\email{sjm1@williams.edu, Steven.Miller.MC.96@aya.yale.edu}
\address{Department of Mathematics, Williams College, MA 01267}



\maketitle

\begin{abstract} A decade ago, when teaching complex analysis, the third named author posed the question on whether or not there is an analogue to the Schwarz lemma for real analytic functions. This led to the note \cite{MT}, indicating that it is possible to have a real analytic automorphism $f$ of $(-1,1)$ with $f'(0)$ arbitrarily large. In this note we provide other families with this property, and moreover show that we can always find such a function so that $f'(0)$ equals any desired real number. We end with some questions on related problems.
\end{abstract}

\tableofcontents

\section{Introduction}

One of the gems of complex analysis is the Riemann Mapping Theorem, which states that any simply connected proper subset of the complex plane $\Omega$ is conformally equivalent to the unit disk $\mathbb{D}$; for details see, for example, \cite{SS}. Explicitly, there is a complex holomorphic map $f:\Omega \to \mathbb{D}$ which is a bijection, and whose inverse is also holomorphic. Much more is true about such functions, as complex differentiable functions are infinitely differentiable and analytic (equalling their power series expansion). One of the key ingredients in the proof is the following.

\begin{lem}[Schwarz lemma] If $f$ is a holomorphic map of the unit disk to itself that fixes the origin, then $|f'(0)| \le 1$; further, if $|f'(0)| = 1$ then $f$ is an automorphism (in fact, a rotation).
\end{lem}

When teaching complex analysis at Williams in Fall 2010, the third named author asked the class whether or not similar results hold for real analytic isomorphisms from $(-1,1)$. With David Thompson \cite{MT} he showed that, as is true with many problems in complex analysis, the real case exhibits very different behavior. In particular, given any $B$ one can find a real analytic isomorphism $f$ with $f'(0) > B$; the example they used was $f_k(x) := {\rm erf}(kx)/{\rm erf}(k)$, where ${\rm erf}$ is the error function (related to the area under normal distributions): $${\rm erf}(x) \ := \ \frac{2}{\sqrt{\pi}} \int_0^x e^{-t^2} dt.$$

In this note we provide two additional families where the derivatives can be arbitarily large at the origin, and show further that we can always find a real analytic automorphism of $(-1, 1)$ with derivative $a$ at the origin. We end with a discussion of related problems.

\section{Iterated Compositions}

The general idea of the construction is the following: if we have a function that is a real analytic automorphism of $(-1,1)$ with derivative exceeding 1, then if we compose it with itself we get another real analytic automorphism whose derivative at the origin is square of the original. Thus if we keep iterating we can make the derivative at zero arbitrarily large. We give a proof at the end of the section that the composition of two real analytic functions is real analytic (Lemma \ref{lem:compositionrealanalytic}).

We state the following lemma in more generality than we need. In our application, the fixed point is the origin.

\begin{lem}
Let $h:(-1,1)\to (-1,1)$ be a real analytic automorphism with a fixed-point $x^\ast$. If we define $h_1 := h$ and $h_n := h\circ h_{n-1}$, then each function $h_n$ is a real analytic automorphism fixing $x^\ast$ and $h'_n(x^\ast) = (h'(x^\ast))^n$.
\end{lem}

\begin{proof}
As $h_n$ is the composition of real analytic automorphisms of $(-1,1)$, it is a real analytic automorphism of $(-1,1)$ (see Lemma \ref{lem:compositionrealanalytic}). Clearly it fixes $x^\ast$. Then, by induction and the chain rule, we have 
\[
h'_n(x^\ast)\ = \ (h \circ h_{n-1})'(x^\ast) \ = \ h'(h_{n-1}(x^\ast)) \cdot h_{n-1}'(x^\ast), \] and by induction $h_{n-1}'(x^\ast) = h'(x^\ast)^{n-1}$, completing the proof.
\end{proof}

Thus to show that the derivative can be arbitarily large at the origin, it suffices to find one real analytic automorphism with derivative exceeding 1. There are many examples we can use; a particularly simple one is $h(x) = \sin(\pi x/2)$, with $h'(0) = \pi/2 > 1$.

First, $h$ is real analytic asit is the composition of the functions $\sin(x)$ and $\pi x/2$, both of which are real analytic everywhere. Second, it is an automorphism of $(-1,1)$ since $h'(x) =\cos(\pi x/2) \cdot  \pi/2$ is positive on the entire interval. As such, $h$ is monotone increasing, hence injective. Moreover, $h(-1) = -1$, $h(1) = 1$, so it is also surjective by the Intermediate Value Theorem. Thus $h$ is an automorphism of $[-1,1]$. If we delete $-1$ and $1$ from the domain of $h$, the image is now missing $h(-1)$ and $h(1)$. Thus $h$ is also an automorphism of $(-1,1)$. 
For the first few iterations of $h_n$, see Figure \ref{fig:compositionexample} \cite{SM}.

\begin{figure}[h]
\begin{center}
\scalebox{0.8}{\includegraphics{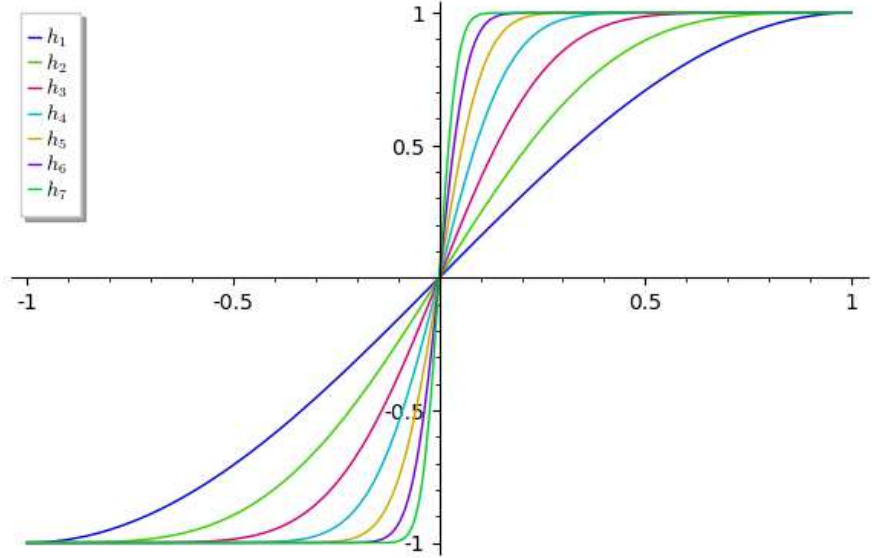}}
\caption{\label{fig:compositionexample} The first few iterations arising from $h_1 = \sin(\pi x/2)$.}
\end{center}
\end{figure}


\ \\

We end this section with the needed result and proof.

\begin{lem}\label{lem:compositionrealanalytic}
Let $f,g$ be real analytic automorphisms of $(-1,1)$. Then the composition $f\circ g$ is also a real analytic automorphism of $(-1,1)$.
\end{lem}

\begin{proof}
Since $g^{-1}\circ f^{-1}$ is clearly a two-sided inverse of $(f\circ g)$, we have $f\circ g$ is an isomorphism of $(-1,1)$. It can be proven in a number of ways that the composition $f\circ g$ must be real analytic, one nice path is through complex analysis.

As the functions are real analytic, we may write $$f(x)\  \ = \sum_{m=0}^\infty a_m x^m  \ \ \ {\rm and}\ \ \  g(x)\ =\ \sum_{n=0}^\infty b_nx^n.$$ We can extend these functions to the unit disk since the sums converge absolutely when $|x|<1$. We then have that $f(z)$ and $g(z)$ are holomorphic functions and therefore so is $(f\circ g)(z)$. In particular, $(f\circ g)(z)$ must equal its Taylor series on $(-1,1)$ as holomorphic functions are analytic. As it can only take real values when $z$ is real (i.e., $z = x$), we find $f\circ g$ is real analytic.
\end{proof}

\section{Arctangent and Attaining All Values}

By considering a family related to the $\arctan$ function, we prove the following.

\begin{thm} For any real $a$ there exists a real analytic automorphism of $(-1, 1)$ with derivative $a$ at the origin. 
\end{thm}

\begin{proof}
It suffices to show that we can attain any non-negative value $a$, as replacing $f(x)$ with $-f(x)$ yields negative values. 

Further, since $f(x) = x^3$ satisfies the conditions and has $f'(0) = 0$, we see we need only attain positive values. Finally, since $f(x) = x$ works and has $f'(0) = 1$, we are reduced to finding an automorphism $f$ such that $f'(0) = a$ for any positive $a \neq 1$. \\ \

\noindent \textbf{Case I: $a > 1$:} Consider functions of the form
\begin{eqnarray}
f_{a,b}(x) & \ = \  & \frac{a}{b}\arctan\left(bx\right) \nonumber\\
f_{a,b}'(x) & = & \frac{a}{1+(bx)^2}\nonumber\end{eqnarray}

Clearly, \(f_{a.b}'(0)=a\). We claim that for any \(a>1\), there exists a positive \(b^\ast\) such that \(f_{a,b^\ast}(1)=1\) and \(f_{a,b^\ast}(-1)=-1\). To see this, note that \(f_{a,b}(1)=1\) if \(a= b / \arctan(b)\). The function \(h(b)= b / \arctan(b)\) takes every value greater than 1 by inspection, so for any \(a>1\), there is a value of \(b\) that guarantees the function takes the value 1 at 1. Since the function is odd, this means it takes the value $-1$ at $-1$. Since the derivative is nowhere zero and the function is continuous everywhere on \([-1,1]\), the function must take every value in \([-1,1]\) at least once for \(x\in[-1,1]\) by the Intermediate Value Theorem, and can never take the same value more than once since it is a strictly increasing function. Putting these together, it must be an automorphism from the unit interval to itself. Additionally, since \(\arctan(x)\) is analytic and has a Taylor series which converges for all \(x\), \(\arctan(bx)\) also must have a Taylor series which converges for all \(x\) when \(b\) is finite, since its Taylor series has a radius of convergence of \(R/b\) if the radius of convergence of the original Taylor series is \(R\). Thus these families of functions give analytic functions with any derivative \(a>1\) at the origin and are automorphisms from \((-1,1)\) to itself.
\begin{figure}
    \centering
    \includegraphics{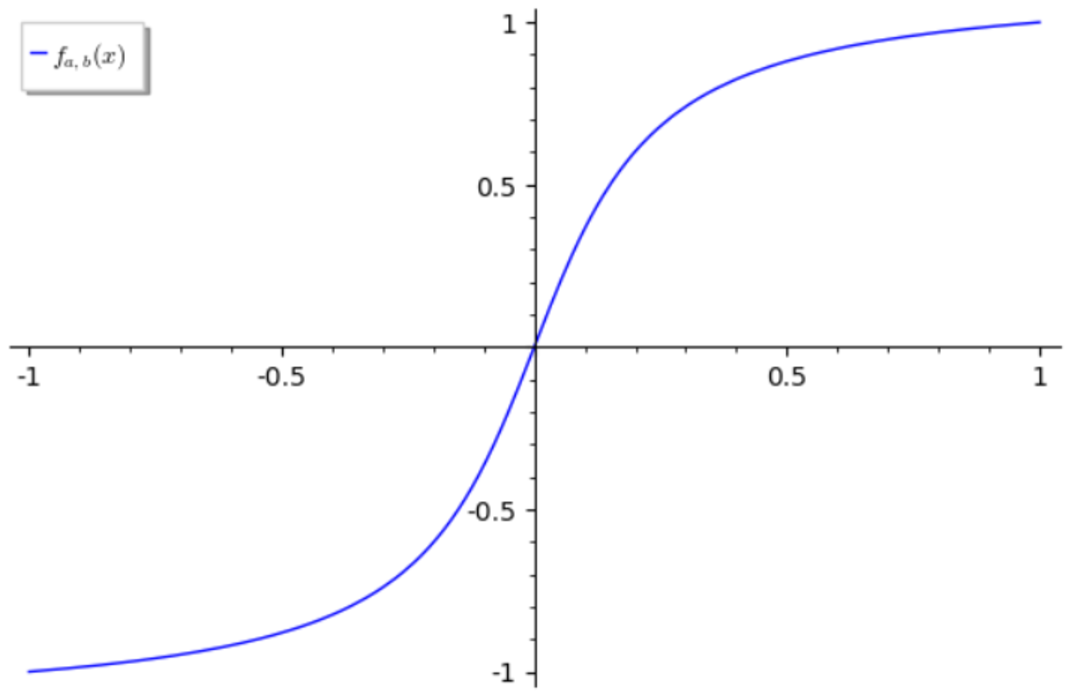}
    \caption{$f_{a,b}$ for $a = 4$ \cite{SM}}
    \label{fig:exact-case1}
\end{figure}\mbox{}\\
\noindent \textbf{Case II: $0 < a < 1$:} Consider 
\begin{eqnarray} g_{a,b}(x) &  \ = \ & \frac{a}{b}\tan(bx)\nonumber\\
g_{a,b}'(x) & =  & a\sec^2(bx) \ \ \ \implies\ \ \  g_{a,b}'(0) \ = \ a\nonumber
\end{eqnarray}

Given any \(0<a<1\), we claim we can select a \(b\) such that \(g_{a,b}(1)=1\) and \(g_{a,b}\) is continuous on \((-1,1)\). To see this, note that \(g_{a,b}(1)=1\) when \(b/\tan(b) =a\). Since \(b/\tan(b)\) takes every value in \((0,1)\) for \(b < \pi/2\) by inspection, we can find a \(b^\ast\) that makes \(g_{a,b^\ast}(1)=1\) for any \(0<a<1\). Also, since \(0 < b < \pi/2\), the function is strictly increasing on \([-1,1]\) and is well-defined and continuous on the interval. Since we have a continuous, strictly increasing function satisfying \(g_{a,b^\ast}(-1)=-1\) and \(g_{a,b^\ast}(1)=1\), by the same arguments applied earlier for \(f_{a,b^\ast}\), this means we have found a family of automorphisms from \((-1,1)\) to itself with derivatives \(0 < a < 1\). Since the Taylor series about \(x=0\) for \(\tan(x)\) has a radius of convergence of \(\pi/2\), the radius of convergence for the Taylor series of \(g_a(x)\) must be \(\pi/2b\). Since \(0 < b < \pi/2\), the radius of convergence is at least one, and so \(g_a\) is analytic on the unit interval. 
\end{proof}
\begin{figure}
    \centering
    \includegraphics{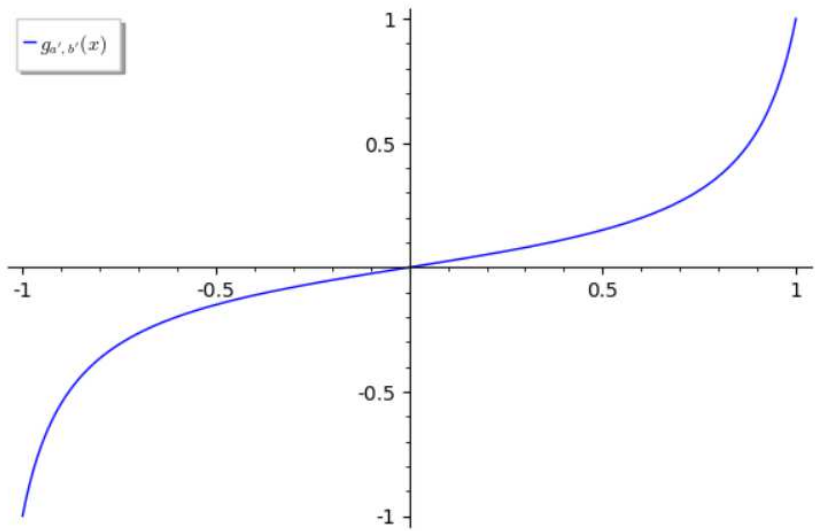}
    \caption{$g_{a',b'}$ for $a' = 1/4$ \cite{SM}}
    \label{fig:exact-case2}
\end{figure}


\section{Comments and Future Work}

The examples here show, yet again, how different the real and complex cases are. One can ask similar questions about other ingredients in the proof of the Riemann Mapping Theorem. For example, another input is that if one has a sequence of injective holomorphic functions that converge uniformly on compact sets to a holomorphic function, then the limit function is either constant or injective. Does a similar result hold in the real setting? 



If we interpret holomorphic as infinitely differentiable, then the answer is no. We consider the function $$\twocase{f(x) \ := \ }{\exp(-1/x^2)}{if $x > 0$}{0}{if $x\leq 0$.}$$ This function is neither constant nor injective, but the sequence of functions $f_n(x):= f(x) + x/n$ are monotone increasing, hence injective everywhere. Restricted to a compact domain, say $[-a, a]$, we have that $f_n\to f$ uniformly. This part of the Riemann mapping theorem will fail also. 


\begin{figure}[h]
\begin{center}
\scalebox{0.8}{\includegraphics{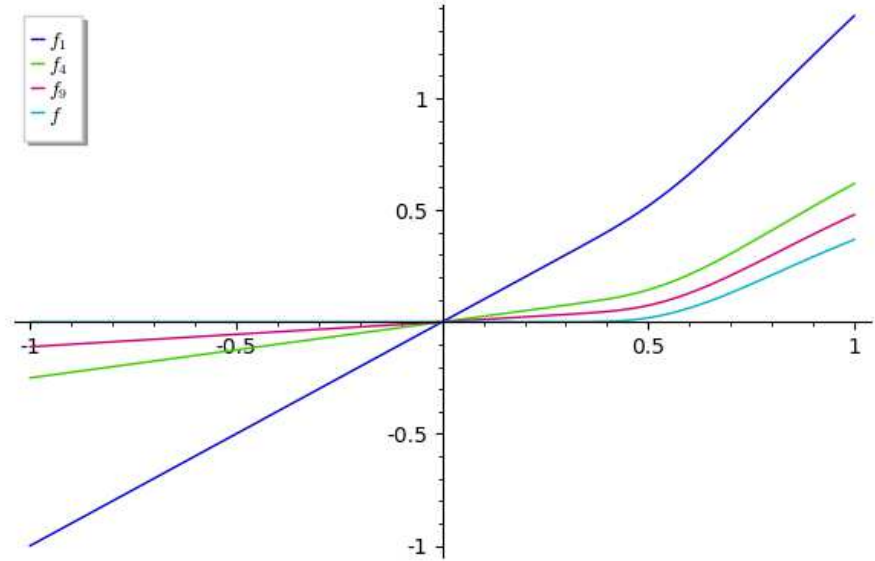}}
\caption{\label{fig:injectiveloss} Plot of the functions $f_1,f_4,f_9,f$ on $[-1,1]$ \cite{SM}. The first three of these are injective, however the fourth is not.}
\end{center}
\end{figure}

For another example (which is differentiable albeit not infinitely differentiable), consider the function  $$\twocase{g_n(x) \ := \ }{-\frac{x^3}{3n} - \frac{x^2}{2n}}{if $-1 < x \le 0$}{\frac{x^2}{2}}{if $0 \le x < 1$,}$$ so that $$\twocase{g_n'(x) \ = \ }{\frac{-x(x+1)}{n}}{if $-1 < x \le 0$}{x}{if $0 \le x < 1$.}$$ Note $g_n$ is continuous and differentiable, and for any finite $n$ we have $g_n'(x) > 0$ except at $x = 0$ and $x=-1$, where it vanishes; see Figure \ref{fig:millernoninjlimit}. As $n \to \infty$ we have $g_n(x)$ converges to $g(x)$, which is given by $$\twocase{g(x) \ := \ }{0}{if $-1 < x \le 0$}{\frac{x^2}{2}}{if $0 \le x < 1$.}$$ Thus we have a sequence of injective functions whose limit is neither injective nor constant. It is worth noting that this would not be a counter-example if we consider $g_n$ and $g$ to be functions of $z$; while the split definition yields a continuous function of a real variable, it does not of a complex one.

\begin{figure}[h]
\begin{center}
\scalebox{0.8}{\includegraphics{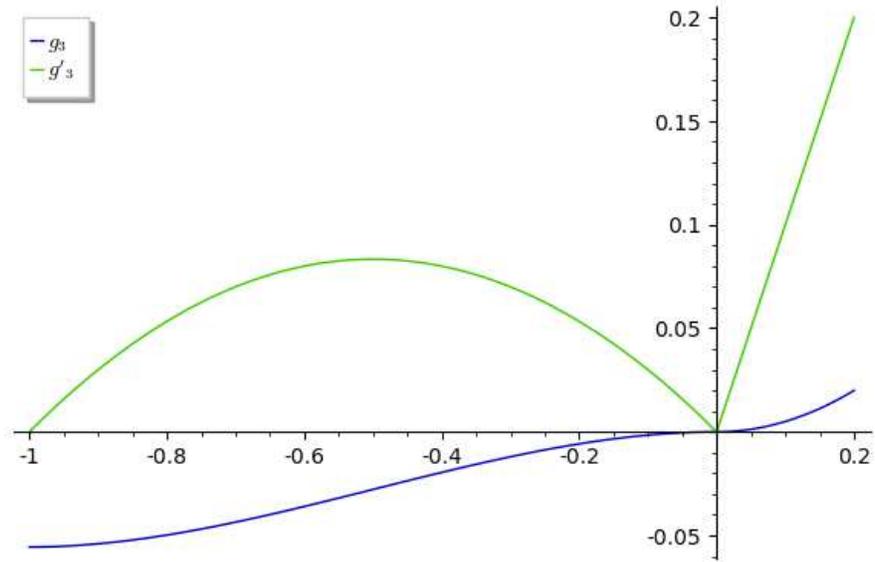}}
\caption{\label{fig:millernoninjlimit} Plot of $g_n, g'_n$ when $n=3$ \cite{SM}}
\end{center}
\end{figure}


For both of our examples, while our functions are differentiable they are not real analytic; \emph{can one find a sequence of real analytic functions which converges uniformly on compact sets whose limit is a non-constant, non-injective function?}


\bibliographystyle{alpha}

\begin{thebibliography}{A999}

\bibitem[MT]{MT} 
S. J. Miller and D. Thompson, \emph{The real analogue of the Schwarz lemma}, American Mathematical Monthly \textbf{118} (October 2011), Number 8, page 725.

\bibitem[SM]{SM}
SageMath, \emph{the Sage Mathematics Software System} (Version 9.1),
The Sage Developers, 2021, \href{https://www.sagemath.org}{SageMath, the Sage Mathematics Software System (Version x.y.z),
   The Sage Developers, YYYY, https://www.sagemath.org.}.

\bibitem[SS]{SS}
E. Stein and R. Shakarchi, \emph{Complex Analysis}, Princeton University Press, Princeton, NJ, 2003.
\end{thebibliography}

\bigskip

\end{document}